\documentclass{daj}
\usepackage[utf8]{inputenc}
\usepackage{amsmath,amsfonts,amsthm,amssymb,verbatim,etoolbox,color}
\usepackage{enumitem}
\usepackage{epigraph}
\newcommand{\sub}{\subseteq}

\newcommand{\bra}[1]{\left(#1\right)}
\newtheorem{theorem}{Theorem}
\newtheorem{lemma}[theorem]{Lemma}
\newtheorem{conjecture}[theorem]{Conjecture}
\newtheorem{proposition}[theorem]{Proposition}
\newtheorem{definition}[theorem]{Definition}
\newtheorem{remark}[theorem]{Remark}
\numberwithin{theorem}{section}

\newcommand{\F}{\mathbb{F}}

\dajAUTHORdetails{%
  title = {New Lower Bounds for Cap Sets}, 
  author = {Fred Tyrrell},
  plaintextauthor = {Fred Tyrrell},
  plaintexttitle = {New Lower Bounds for Cap Sets}, 
  keywords = {cap sets, finite field, arithmetic progressions, SAT solver},
}   

\dajEDITORdetails{%
   year={2023},
   number={20},
   received={23 September 2022},   
   published={13 December 2023},  
   doi={10.19086/da.91076},       
}   

\begin{document}
\renewcommand{\imageat}{\includegraphics[width=3mm]{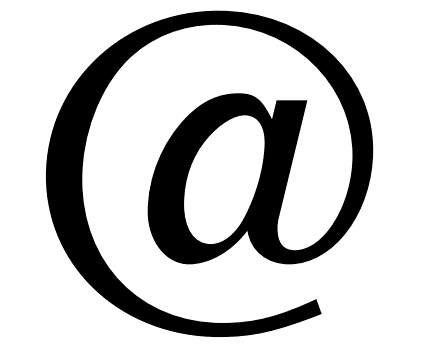}}
\renewcommand{\imagedot}{\includegraphics[width=1mm]{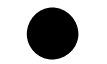}}

\begin{frontmatter}[classification=text]


\author[Fred]{Fred Tyrrell\thanks{Supported by Thomas Bloom's Royal Society University Research Fellowship}}

\begin{abstract}
A cap set is a subset of $\mathbb{F}_3^n$ with no solutions to $x+y+z=0$ other than when $x=y=z$. In this paper, we provide a new lower bound on the size of a maximal cap set. Building on a construction of Edel, we use improved computational methods and new theoretical ideas to show that, for large enough $n$, there is always a cap set in $\mathbb{F}_3^n$ of size at least $2.218^n$.
\end{abstract}
\end{frontmatter}

\section{Introduction}
\begin{definition}\label{def:cap}
A \emph{cap set} is a set $A \sub \F_3^n$ with no solutions to $x+y+z=0$ other than when $x=y=z$, or equivalently a set $A$ with no 3 distinct elements in arithmetic progression.
\end{definition}
In this paper, we prove the following result.
\begin{theorem} \label{main}
There is a cap set in $\mathbb{F}_3^{56232}$ of size 
\[\binom{11}{7}^{141} \cdot 6^{572} \cdot 12^{572} \cdot 112^{8800} \cdot 37 \cdot 142\]
and hence, for large $n$, there is a cap set $A \sub \F_3^n$ of size $(2.218021\ldots)^n$.
\end{theorem}
Since \cite{Pellegrino} in 1970, there have been two improvements to the lower bound on the size of a maximal cap set. A lower bound of $(2.210147\ldots)^n$ was given by Calderbank and Fishburn \cite{CalderbankFishburn}, and then Edel improved this to $(2.217389\ldots)^n$ in \cite{Edel}. In this paper, we obtain the first new lower bound for nearly two decades, and prove that a maximal cap set has size at least $(2.218021\ldots)^n$. Our method is based on that of Edel, which in turn is based on the work of Calderbank and Fishburn, and involves taking cap sets which are known to be maximal in a low dimension and then combining them carefully to produce large cap sets in higher dimensions.
\bigbreak
The upper bound has received significant attention, both historically and recently. In \cite{Meshulam}, an upper bound of $\frac{3^n}{n}$ was shown, which was then improved to $\frac{3^n}{n^{1+\varepsilon}}$ for some $\varepsilon > 0$ in \cite{BK}.
It was an open problem for over 20 years as to whether a cap set in $\mathbb{F}_3^n$ has size at most $c^n$, for some $c<3$. This was finally solved by Ellenberg and Gijswijt in \cite{EllenbergGijswijt}, who used the polynomial method developed by Croot, Lev and Pach in \cite{CLP} to show that a cap set in $\mathbb{F}_3^n$ has size at most $(2.7552\ldots)^n$.
A symmetric version of the polynomial method proof for the upper bound has since been formulated by Tao in \cite{tao_2016}, and Ellenberg and Gijswijt's result was formalised in the Lean theorem prover in \cite{dahmen2019formalizing}. There is also an improvement of the upper bound by an extra factor of $\sqrt{n}$ in \cite{jiang2021improved}.

\bigbreak
One major reason for the interest in cap sets is that they can provide useful insights into similar problems in more complicated sets. For example, the recent breakthrough on Roth's theorem due to Kelley and Meka \cite{kelley2023strong} and the previous work of Bloom and Sisask \cite{bloom2020} make use of several ideas from the upper bound for the cap set problem given by Bateman and Katz in \cite{BK}. An alternative exposition of the Kelley-Meka bound, written for those with a background in additive combinatorics, is given in \cite{bloom2023kelley}. One may therefore think of the cap set problem as a `finite field model' of the question of finding the largest progression-free set of $\{1, \ldots, N\}$.
\smallbreak
Cap sets are also of interest in finite geometry, design of experiments and various other problems in combinatorics and number theory. For an excellent survey on the motivation and background of the problem, and its application to several other interesting questions, we recommend the article \cite{grochow2019new}.
\subsection*{Structure of the paper}
In section 2, we present the extended product construction of Edel, and combine it with improved computational techniques to obtain some new lower bounds. We then introduce a new construction in section 3, which takes the extended product from section 2 up a level, and use this to achieve the best lower bound in this paper. In section 4, we discuss possible ideas for future work, as well as the limitations of our approach. Finally, we describe our computational methods, which make use of a SAT solver, in section 5.

\section{The Extended Product Construction}
In this section, we describe a construction due to Edel in \cite{Edel}, which extends the ideas from \cite{CalderbankFishburn} to produce larger cap sets. We use Edel's method to construct a cap set in 396 dimensions which gives a lower bound of $(2.217981\ldots)^n$, already an improvement on Edel's lower bound of $(2.217389\ldots)^n$. We first describe a rather simple construction.
\begin{proposition}[Product Caps] \label{def:ProdCap}
Let $A \sub \mathbb{F}_3^n$, $B \sub \mathbb{F}_3^m$ be cap sets. Then, by taking a direct product of $A$ and $B$, there is a cap set of size $|A||B|$ in $\mathbb{F}_3^{n+m}$.
\end{proposition}
\begin{proof}
Define $A \times B = \{(a,b) : a \in A, b \in B\}$. Clearly $|A \times B| = |A||B|$, and we show that $A \times B$ is a cap set.
\smallbreak
Assume we have a solution to $x+y+z = 0$ in $A \times B$. So $x = (x_a, x_b)$, $y = (y_a, y_b)$ and $z = (z_a, z_b)$ where $x_a, y_a, z_a \in A$ and $x_b, y_b, z_b \in B$. Therefore, $x_a + y_a + z_a = 0 = x_b + y_b + z_b$. Since $A,B$ are both cap sets, we must have $x_a = y_a = z_a$ and $x_b = y_b = z_b$, so $x=y=z$. Hence $A \times B$ has no non-trivial solutions to $x+y+z = 0$, and is therefore a cap set.
\end{proof}
\smallbreak
The above proposition shows that there is always a cap set in $\F_3^n$ of size $2^n$, by taking direct products of the set $\{0,1\} \sub \F_3$, to form the cap set $\{0,1\}^n \sub \F_3^n$. We now use the direct product construction to show how we can derive an asymptotic lower bound for the cap set problem.
\begin{proposition} \label{prop:lower bound asymptotic}
Let $A \sub \mathbb{F}_3^n$ be a cap set of size $c^n$. Then for any $\varepsilon > 0$, there is an $M$ such that for all $m \geq M$, there is a cap set of size greater than $\left(c - \varepsilon\right)^{m}$ in $\mathbb{F}_3^m$.
\end{proposition}
\begin{proof}
Let $A \sub \mathbb{F}_3^n$ be a cap set of size $c^n$.
For $m > n$, there is $k$ such that $m = nk + r$, where $0 \leq r<n$. Applying the product construction $k$ times to $A$, we have a cap set in $\mathbb{F}_3^m$ of size $c^{nk} = \left(c^{1-r/m}\right)^m$. 
\smallbreak
Given $\varepsilon>0$, we can always choose a large enough $M$ such that $c^{1-n/M} > c - \varepsilon$. Since $r<n$, for all $m \geq M$ we have $\frac{r}{m}< \frac{n}{M}$, and hence $c^{1-r/m} > c - \varepsilon$. Therefore, we have constructed a cap set in $\mathbb{F}_3^m$ of size greater than $(c - \varepsilon)^m$.
\end{proof}
\begin{remark}
This proposition demonstrates that finding an asymptotic lower bound for the cap set problem amounts to finding a cap set $A \sub \mathbb{F}_3^n$ where $|A|^{1/n}$ is as large as possible. If $A \sub \mathbb{F}_3^n$ and $B \sub \mathbb{F}_3^m$ are cap sets, we can use the previous proposition to achieve a lower bound of $|A|^{1/n}$ and $|B|^{1/m}$ respectively. We know we can form the direct product $A \times B$, but since $|A \times B|^{\frac{1}{n+m}} \leq \text{max}\left(|A|^{\frac{1}{n}},|B|^{\frac{1}{m}}\right)$, the bound from $A \times B$ will never beat the better of the bounds from $A$ and $B$.
\end{remark}
In \cite{Edel}, Edel introduced a new construction based on the work in \cite{CalderbankFishburn}, which one can think of as a sort of twisted product. The idea is simple - if we start with a collection of cap sets, and construct several different direct products, can we take the union of the direct products and still be a cap set? The answer is yes, under certain conditions on the cap sets and the way we combine them.

\begin{definition}[Extendable collection] \label{def:extendable}
Let $A_0, A_1, A_2 \sub \mathbb{F}_{3}^{n}$ be cap sets. We say that this collection of cap sets is \emph{extendable} if the following 2 conditions hold:
\begin{enumerate}
    \item If $x,y\in A_0$ and $z\in A_1\cup A_2$ then $x+y+z \neq 0$.
    \item If $x\in A_0$, $y\in A_1$ and $z\in A_2$ then $x+y+z\neq 0$.
\end{enumerate}
Note that taking $x=y$ in condition (1) shows that $A_0$ is disjoint from $A_1$ and $A_2$.
\end{definition}
\smallbreak

\begin{definition}[Admissible set\footnote{To avoid confusion, we note that our terminology and definitions are not the same as those in \cite{Edel}. In particular, the object which Edel calls a `cap' we call a `cap set', and Edel's `capset' is our `admissible set'.}] \label{def:admissible}
Let $S \sub \{ 0,1,2\}^m$. $S$ is \emph{admissible} if:
\begin{enumerate}
    \item For all distinct $s, s' \in S$, there are coordinates $i$ and $j$ such that $s_i = 0 \neq s_i'$ and $s_j \neq 0 = s_j'$.
    \item For all distinct $s, s', s'' \in S$, there is a coordinate $k$ such that $\{s_k, s_k', s_k''\} = \{0, 1, 2\}$, $\{0, 0, 1\}$  or $\{0, 0, 2\}$.
\end{enumerate}
\end{definition}

\begin{definition}[Extended product construction] \label{def:extended product} As the name suggests, we can extend an extendable collection of cap sets by an admissible set. The construction is as follows:
for $s = (s_1, \ldots, s_m) \in \{0,1,2\}^m$ and $A_0, A_1, A_2 \sub \mathbb{F}_{3}^{n}$, we define \[s(A_0, A_1, A_2) = A_{s_1} \times \cdots \times A_{s_m} \sub \mathbb{F}_{3}^{nm}.\]
If $S \sub \{0,1,2\}^{m}$ is an admissible set, we define \[S(A_0, A_1, A_2) = \bigcup_{s \in S} \ s\left(A_0, A_1, A_2\right) \sub \mathbb{F}_{3}^{nm}.\]
\end{definition}

The following lemma, which although rather different in presentation is essentially Lemma 10 of \cite{Edel}, demonstrates the usefulness of these definitions.

\begin{lemma} \label{lemma:extended product cap} If $(A_0, A_1, A_2)$ is an extendable collection of cap sets in $\mathbb{F}_{3}^{n}$, and $S \sub \{ 0,1,2\}^m$ is an admissible set, then $S(A_0, A_1, A_2)$ is a cap set in $\mathbb{F}_{3}^{nm}$.
\end{lemma}
\begin{proof}
We want to prove that 
\[\bigcup_{s \in S} s(A_0, A_1, A_2) \]
is a cap set, where $s(A_0,A_1, A_2)= A_{s_1} \times \cdots \times A_{s_{m}}$.
\smallbreak
Suppose we have distinct $x,y,z \in S(A_0, A_1, A_2)$ such that $x+y+z=0$. We have 3 cases, depending on where $x,y,z$ come from.
\smallbreak
\textbf{Case 1:} By the direct product construction, we know that each $s(A_0,A_1, A_2)$ is a cap set. So, there are no distinct $x,y,z \in s(A_0,A_1, A_2)$ such that $x+y+z=0$.
\smallbreak
\textbf{Case 2:} Suppose we have $x,y\in s(A_0,A_1, A_2)$ and $z \in s'(A_0,A_1, A_2)$, where $s \neq s'$. So $x=(x_{s_1},\ldots,x_{s_m})$, $y=(y_{s_1},\ldots,y_{s_m})$ and $z=(z_{s_1'},\ldots,z_{s_m'})$. By property (1) of being admissible, there is some coordinate $j$ in which $s_j=0$ and $s_j'\neq 0$. So $x_{s_j}+y_{s_j}+z_{s'_j}=0$ where $x_{s_j},y_{s_j}\in A_0$ and $z_{s'_j} \in A_1 \cup A_2$, contradicting property (1) of extendable. 
\smallbreak
\textbf{Case 3:} Suppose $x,y,z$ come from the distinct vectors $s,s',s''$. By condition (2) of admissible, there is a coordinate $k$ such that $\{s_k, s_k', s_k''\}$ is $\{0, 1, 2\}$, $\{0, 0, 1\}$ or $\{0, 0, 2\}$. If $\{s_k,s_k',s_k''\}=\{0,0,1\}$ or $\{0,0,2\}$, then we have a contradiction of property (1) of extendable, as above. Otherwise, if $\{s_k,s_k',s_k''\}=\{0,1,2\}$, we have a contradiction of property (2) of extendable.
\end{proof}

We will now discuss some important types of admissible sets, which will be used in our later constructions.

\begin{definition}[Recursively admissible set\footnote{Again, note that our terminology is different to Edel's. In \cite{Edel}, these objects are simply called `admissible sets'.}] \label{def:recursive}
$S$ is a \emph{recursively admissible} set if $S$ is an admissible set, $|S| \geq 2$ and for all distinct pairs $s, s' \in S$ at least one of the following holds:
\begin{enumerate}[label=(\roman*)]
    \item There are coordinates $i, j$ such that $\{s_i, s_i'\} = \{0, 1\}$ and $\{s_j, s_j'\} = \{0, 2\}$.
    \item There is a coordinate $k$ such that $s_k = s_k' = 0$.
\end{enumerate}
\end{definition}

Given the name `recursively admissible', the reader may not be too surprised by the flavour of the following lemma.
\begin{lemma} \label{lemma:recursive}
If $(A_0, A_1, A_2)$ is an extendable collection of cap sets, and $S \sub \{0,1,2\}^m$ is a recursively admissible set, then $\left(S\left(A_0, A_1, A_2\right), A_1^m, A_2^m\right)$ is an extendable collection of cap sets.
\end{lemma}
\begin{proof}
$S(A_0, A_1, A_2)$ is a cap set by \eqref{lemma:extended product cap}, and takes the place of $A_0$ from the definition of extendable. First of all, we show that there are no $x,y\in S(A_0, A_1, A_2)$ and $z\in A_1^m\cup A_2^m$ such that $x+y+z=0$.
\bigbreak
Assume $x,y \in S(A_0, A_1, A_2)$ and $z \in A_1^m \cup A_2^m$. If $x,y \in s(A_0, A_1, A_2)$, then since $|S|\geq 2$, it must be that $s$ has at least one coordinate 0.  Let $s_k= 0$, and let the corresponding blocks of $x,y,z$ be $x_k, y_k, z_k$ respectively. Then $x_k, y_k \in A_0$, $z_k \in A_1 \cup A_2$, so by property (1) of $(A_0, A_1, A_2)$ being extendable, $x_k + y_k + z_k \neq 0$.
\smallbreak
Now assume $x \in s(A_0, A_1, A_2)$ and $y \in s'(A_0, A_1, A_2)$. Since $S$ is recursively admissible, either there is a coordinate $k$ such that $s_k = s'_k = 0$ or there are coordinates $j,k$ such that $\{s_j, s'_j\} = \{0,1\}$ and $\{s_k, s'_k\} = \{0,2\}$.
\smallbreak
In the first case where $s_k = s'_k = 0$, we have $x_k, y_k \in A_0$ and $z_k \in A_1 \cup A_2$, so by the same reasoning as above $x_k + y_k + z_k \neq 0$ by property (1) for extendable.
\smallbreak
In the case where $\{s_j, s'_j\} = \{0,1\}$ and $\{s_k, s'_k\} = \{0,2\}$, either $z_j = z_k = 1$ or $z_j = z_k = 2$. Assume that $z_j = z_k = 2$ and $s_j =1$. Then $x_j \in A_1$, $y_j \in A_0$ and $z_j \in A_2$. By property (2) of extendable, we deduce that $x_j + y_j + z_j \neq 0$. 
\smallbreak
Similarly, if $z_j = z_k = 2$ and $s'_j = 1$, then $x \in A_0$, $y \in A_1$ and $z \in A_2$, so $x_j+y_j+z_j \neq 0$ by property (2) of extendable again. Finally, if $z_j  = z_k = 1$, then either $s_k = 0$ and $s'_k = 2$, or $s_k = 2$ and $s'_k = 0$, and once again we use property (2) of extendable to show that $x_k + y_k + z_k \neq 0$.
\smallbreak
Hence, we can never have $x,y\in S(A_0, A_1, A_2)$ and $z\in A_1^m\cup A_2^m$ such that $x+y+z=0$, so condition (1) of extendable holds.
\bigbreak
Now we want to show that if $x\in S(A_0, A_1, A_2)$, $y\in A_1^m$ and $z\in A_2^m$ then $x+y+z\neq 0$. This is relatively straightforward - by the same reasoning as above, $s$ has a coordinate $k$ where $s_k = 0$. So, $x_k \in A_0$, $y_k \in A_1$ and $z_k \in A_2$. Then, by property (2) of extendable, $x_k + y_k + z_k \neq 0$, so $x + y + z \neq 0$. 

\smallbreak
So we have condition (2) for extendable, and hence $\left(S\left(A_0, A_1, A_2\right), A_1^m, A_2^m\right)$ is an extendable collection of cap sets, as required.
\end{proof}

In addition to recursively admissible sets, we will also be making use of admissible sets where every element has the same weight. By `weight', we mean the number of non-zero entries in a vector. We will also talk about the `support' of a vector, meaning the set of non-zero coordinates.
\begin{definition}[Constant weight admissible sets]
Write $S = I(m,w)$ if $S \sub \{0,1,2\}^{m}$ is an admissible set consisting of $\binom{m}{w}$ vectors, each of weight $w$. If in addition $S$ is recursively admissible, write $S = \tilde{I}(m,w)$.\footnote{Once again, we mention the difference between our notation and that of Edel. In \cite{Edel}, the second parameter is not the weight $w$ of each vector, but the number of zeroes $t$. Note that $t = m -w$, and so $\binom{m}{w}$ = $\binom{m}{t}$. It is also worth pointing out that the meanings of $I$ and $\tilde{I}$ are swapped in \cite{Edel} - whereas we use $\tilde{I}$ to denote a stronger property than $I$, Edel uses $\tilde{I}$ to denote the weaker property of being an admissible set (which is referred to there as a capset), and uses $I$ for a recursively admissible set (which Edel simply calls an admissible set).}
\end{definition}

\smallbreak
The good thing about this class of admissible sets is that they satisfy the pairwise condition for admissible automatically. There are $\binom{m}{w}$ different ways to choose $w$ of the $m$ coordinates to be non-zero, and any 2 distinct vectors $x,y$ then necessarily have coordinates $i,j$ such that $x_i = 0 \neq y_i$ and $x_j \neq 0 = y_j$. This will turn out to be helpful when we need to find admissible sets later, as we can start with the $\binom{m}{w}$ different support sets, which immediately gives us the first condition for admissible, and then we can view the second condition as a 2-colouring problem on the non-zero entries of each vector.
\smallbreak
Another reason this type of admissible set is useful to work with is that it makes calculating the size of the extended product cap sets relatively simple.
\begin{lemma} \label{lemma:size}
If we extend a collection of cap sets $(A_0, A_1, A_2)$ by $S = I(m,w) \sub \{ 0,1,2\}^m$, where $|A_1| = |A_2|$, then \[|S(A_0, A_1, A_2)| = \binom{m}{w}|A_0|^{m-w} |A_1|^{w}  .\]
\end{lemma}
\begin{proof}
Let $s \in S$. If $s_i = 1$ or $s_i = 2$, then $|A_{s_i}| = |A_1|$, and if $s_i = 0$ then $|A_{s_i}| = |A_0|$. Recall that we defined $s(A_0, A_1, A_2) = A_{s_1} \times \cdots \times A_{s_m}$. Since there are $m-w$ zero coordinates and $w$ non-zero coordinates in $s$, a simple counting argument gives $|s(A_0, A_1, A_2)| = |A_0|^{m-w} |A_1|^{w}$. Then, as $S(A_0, A_1, A_2) = \bigcup\limits_{s \in S} s(A_0, A_1, A_2)$, the $s(A_0, A_1, A_2)$ are all disjoint and $|S| = \binom{m}{w}$, it follows that $|S(A_0, A_1, A_2)| = \binom{m}{w}|A_0|^{m-w} |A_1|^{w}$.
\end{proof}

It is finally time for some examples of admissible sets. For our first example, we prove the existence of an important family of recursively admissible sets, by a relatively simple construction due to Edel. We then use a computer search to produce particular examples of admissible sets.

\begin{lemma} \label{lemma:m-1}
For any $m \geq 2$, there exists a recursively admissible set $\tilde{I}(m,m-1)$.
\end{lemma}
This is a special case of Lemma 13 in \cite{Edel}, when $c=2$.
\begin{proof}
Construct a set $S$ as follows: consider the $m$ vectors who have exactly one coordinate $0$, and all others non-zero. For each vector, let all entries before the $0$ be $1$, and all entries after the $0$ be $2$. We show S is a recursively admissible set.
\smallbreak
Let $x,y$ be distinct elements of $S$. Then they must have zeroes in different coordinates, so the pairwise condition for admissible holds: there are $i,j$ such that $x_i = 0 \neq y_i$ and $x_j \neq 0 = y_j$. 
\smallbreak
Without loss of generality, let $i<j$. Then $y_i = 1$ and $x_j = 2$ by our construction, so we also have the condition for recursively admissible. That is, there exist $i,j$ such that $\{x_i, y_i\} = \{0,1\}$ and $\{x_j, y_j\} = \{0,2\}$.
\smallbreak
Let $x,y,z$ be distinct elements of $S$, with zero in coordinate $i,j,k$ respectively. Without loss of generality, let $i<k<j$. Then $x_k = 2$, $y_k = 1$, $z_k = 0$, so we have a coordinate $k$ such that $\{x_k, y_k, z_k\} = \{0,1,2\}$, which gives the triples condition for admissible.
\end{proof}

\begin{lemma} \label{computer caps}
There exist admissible sets $I(11,7)$, $I(11,6)$ and $I(10,6)$.
\end{lemma}

The admissible sets were found with a computer search, and can be found on the author's webpage at \url{http://fredtyrrell.com/cap-sets}. The computational methods we employed, including the use of a SAT solver, are described in section 5. Several other admissible sets were given by Edel, including the $I(10,5)$ used to find the previous lower bound in \cite{Edel}. The admissible sets mentioned in \cite{Edel} can be found on Edel's webpage \cite{EdelSets}.

\bigbreak
We now describe an example of an extendable collection of cap sets in $\mathbb{F}_3^6$. This is exactly the collection used in \cite{Edel} and \cite{CalderbankFishburn}, presented slightly differently. We will summarise the construction - for more details, see Section 3 of \cite{Edel} or Section 2, Figure 3 of \cite{CalderbankFishburn}.
\begin{lemma} \label{extendable caps}
There is an extendable collection $(A_0, A_1, A_2)$ of cap sets in $\mathbb{F}_3^6$, where $|A_0| = 12$, $|A_1| = |A_2| = 112$.
\end{lemma}
\begin{proof}
\textbf{Constructing the collection:}
\smallbreak
Consider the following $6 \times 10$ matrix:
\begin{center}
$\begin{pmatrix}
1 & 1 & 1 & 1 & 1 & 0 & 0 & 0 & 0 & 0\\
1 & 1 & 0 & 0 & 0 & 1 & 1 & 1 & 0 & 0\\
1 & 0 & 1 & 0 & 1 & 0 & 0 & 0 & 1 & 1\\
0 & 1 & 0 & 1 & 0 & 0 & 1 & 0 & 1 & 1\\
0 & 0 & 1 & 0 & 1 & 0 & 1 & 1 & 1 & 0\\
0 & 0 & 0 & 1 & 1 & 1 & 0 & 1 & 0 & 1\\
\end{pmatrix}$    
\end{center}
This is the incidence matrix of a $(6,3,2)$-design, an example of a balanced incomplete block design, meaning any pair of rows are both 1 in exactly two coordinates. Let $D \sub \mathbb{F}_{3}^{6}$ be the vectors with non-zero entries given in the coordinates corresponding to the 1s in the matrix. By this, we mean $D$ is the set of vectors whose non-zero coordinates are 123, 124, 135, 146, 156, 236, 245, 256, 345 or 346. 
\smallbreak Let $D'$ be the remaining vectors of $\mathbb{F}_{3}^{6}$ with three non-zero entries. There are $\binom{6}{3} \times \ 2^3 = 160$ vectors of weight 3, $|D|$ = $2^3 \times 10 = 80$, so $|D| = |D'| = 80$.
\smallbreak Let $R$ be the vectors with no zeros, and an even number of 1s, let $R'$ be the other weight 6 vectors, with an odd number of 1s. $|R| = |R'| = 32$. 
\smallbreak Define $A_1 = D \cup R$, $A_2 = D' \cup R$, and let $A_0$ be the vectors of weight 1. Then $A_0$ is a cap set of size 12 and $A_1, A_2$ are cap sets of size 112 in $\mathbb{F}_{3}^{6}$. Furthermore, $|A_1 \cap A_2|$ = 32 and $A_1 + A_2 = \mathbb{F}_{3}^{6} \setminus A_0$. This all follows from the properties of a block design, and by checking that $D + D'$, $R+R$, $D+R$ and $D' + R$ don't contain any weight 1 vectors.
\bigbreak
\textbf{This collection is extendable:}
\smallbreak
Let $x,y \in A_0$. Since all elements of $A_0$ have weight 1, $x+y$ must have weight $0, 1$ or $2$. If $z \in A_1 \cup A_2$ is such that $x + y + z = 0$, then $z$ needs to have the same weight as $x+y$. Since $A_1 \cup A_2$ consists only of vectors of weights 3 or 6, we cannot have such a $z$. So, there are no solutions to $x+y+z=0$ where $x,y \in A_0$ and $z \in A_1 \cup A_2$, which is condition (1) for extendable.
\smallbreak
Let $x \in A_1$ and $y \in A_2$. Since $A_1 + A_2 = \mathbb{F}_{3}^{6} \setminus A_0$, and $z \in A_0 \iff 2z \in A_0$, there is no $z \in A_0$ such that $x + y + z = 0$. This is condition (2) for extendable, and so we have an extendable collection of cap sets.
\end{proof}

We are now ready to prove the main result of this section, and obtain our first new lower bound for maximal cap sets. We will use several results and examples from this section, to show the following.
\begin{theorem} \label{thm:396}
There exists a cap set $A \sub \mathbb{F}_3^{396}$ of size
\[\binom{11}{7} \cdot 6^4 \cdot 12^4 \cdot 112^{62}.\]
\end{theorem}
\begin{proof} First, we take the extendable collection $(A_0, A_1, A_2)$ from \eqref{extendable caps}, and extend it by the recursively admissible set $S = \tilde{I}(6,5)$, which exists by \eqref{lemma:m-1}. This gives a cap set $B \sub \mathbb{F}_{3}^{36}$ by \eqref{lemma:extended product cap}, where $|B| = 6   \times 112^5 \times 12$ by \eqref{lemma:size}, and an extendable collection $(B, A_{1}^6, A_{2}^6)$ by \eqref{lemma:recursive}.
\smallbreak
Then we extend our new collection $(B, A_{1}^6, A_{2}^6)$ by $T = I(11,7)$ from \eqref{computer caps}, to produce a cap set in $\mathbb{F}_3^{396}$, which by \eqref{lemma:size} has size \[\binom{11}{7} \cdot (6 \times 112^5 \times 12)^4 \cdot (112^6)^7 .\]
\end{proof}
\begin{remark} \label{otherbounds}
Since $|A|^{1/396} \approx 2.217981$, we see that our cap in $\mathbb{F}_3^{396}$ gives an exponential improvement on the asymptotic lower bound on the size of a cap set. We note that Edel's lower bound uses essentially the same method, but with the admissible sets $S=\tilde{I}(8,7)$ and $T=I(10,5)$. Using either $I(10,6)$ or $I(11,6)$ from \eqref{computer caps} produces a better lower bound than Edel, but neither is better than our new bound of $\approx 2.217981^n$.
\end{remark}

\section{Extending The Admissible Set Construction}
In this section, we will extend Edel's methods, by mimicking the extended product for cap sets to find large admissible sets. This will allow us to construct large admissible sets, much larger than is computationally feasible, which will be used to obtain a lower bound of $2.218^n$. 
\smallbreak
The following proposition is similar to \eqref{def:ProdCap}, where we showed that the direct product of cap sets is a cap set.
\begin{proposition} \label{prop: ad prod}
If $S, T$ are admissible sets, then so is their direct product $S \times T$.
\end{proposition}
\begin{proof}
\textbf{Pairwise condition:}
\\Let $a,b \in S \times T$ be distinct. Then $a,b$ are of the form $(s,t)$, $(s',t')$ for some $s,s' \in S$ and $t,t' \in T$. If $s,s' \in S$ are distinct, there are coordinates $i,j$ such that $s_i = 0 \neq s_i'$ and $s_j \neq 0 = s_j'$ by condition (1) of admissibility for S. The same is true for $t,t' \in T$, and since $(s,t) \neq (s',t')$, we must have $s \neq s'$ or $t \neq t'$. It follows that for all $a \neq b \in S \times T$, we have coordinates $i,j$ such that $a_i = 0 \neq b_i$ and $a_j \neq 0 = b_j$. So condition (1) for admissibility is satisfied.
\bigbreak
\textbf{Triples condition:}
\\Let $a,b,c \in S \times T$ be distinct. As before, we must have $a = (s,t)$, $b = (s',t')$ and $c = (s'',t'')$ for some (not necessarily distinct) $s,s',s'' \in S$ and $t,t',t'' \in T$. 
    \\ \textbf{Case 1:} If $s,s',s''$ or $t,t',t''$ are all distinct, then condition (2) of admissibility follows immediately by the admissibility of $S$ or $T$. 
    \\ \textbf{Case 2:} Assume that neither $s,s',s''$ nor $t,t',t''$ are all distinct. Without loss of generality, we may assume $s = s'$. Since $(s,t) \neq (s',t')$, we cannot also have $t=t'$, but since $t,t',t''$ are not all distinct, without loss of generality $t=t''$. So, we must have: $a = (s,t)$, $b = (s,t')$ and $c = (s'',t)$. Since $a,b,c$ are distinct, we must have $s \neq s''$. By condition (1) of admissibility of S, there is a coordinate $k$ such that $s_k = 0 \neq s''_k$. So, $\{a_k, b_k, c_k\} = \{0,0,1\}$ or $\{0,0,2\}$, hence condition (2) for admissibility also holds for $S \times T$.
\end{proof}
\begin{remark}
Now we have a direct product construction for admissible sets, there are at least 2 natural questions:
\begin{enumerate}
    \item Does the direct product construction allow us to produce better admissible sets, by combining known admissible sets?
    \item Can we generalise the direct product construction for admissible sets, in an analogous way to the extension construction \eqref{def:extended product} for cap sets?
\end{enumerate}
\end{remark}
The answer to the first question is no: much like the situation for direct products of cap sets, taking a direct product of admissible sets only ever does as well as the best of the individual admissible sets.
\smallbreak
We focus on the second question - can we improve the product construction for admissible sets, in a similar way to the extended product construction for cap sets? We will answer this question in the affirmative. In particular, in this section we will construct an admissible set in $\{0,1,2\}^{1562}$, which we use to prove \eqref{main} and obtain the lower bound of $2.218^n$.

\begin{definition}[Meta-admissible] \label{def:meta ad}
We say a set $T \sub \{0, 1,2\}^r$ is \emph{meta-admissible} if it is admissible, as in \eqref{def:admissible}. Recall that admissible means:
\begin{enumerate}
    \item For all distinct $t, t' \in T$, there are coordinates $i$ and $j$ such that $t_i = 0 \neq t_i'$ and $t_j \neq 0 = t_j'$.
    \item For all $t, t', t'' \in T$ distinct, there is a coordinate $k$ such that $\{t_k, t_k', t_k''\} = \{0, 1, 2\}$, $\{0,0,1\}$ or $\{0, 0, 2\}$.
\end{enumerate}
\end{definition}
\begin{definition}[Meta-extendable] \label{def:meta ext}
A collection $S_0, S_1, S_2 \sub \{0,1,2\}^m$ of admissible sets is said to be \emph{meta-extendable} if:
\begin{enumerate}
    \item For any $s \in S_0$ and $s' \in S_1 \cup S_2$, the weight of $s$ is less than the weight of $s'$, so all the vectors in $S_0$ have more zeroes than any vector in $S_1$ or $S_2$.
    \item If $x,y \in S_0$ and $z \in S_1 \cup S_2$ then there is a coordinate $k$ such that $\{x_k, y_k, z_k\} = \{0, 1, 2\}$, $\{0,0,1\}$ or $\{0, 0, 2\}$.
    \item If $x \in S_0$, $y \in S_1$ and $z \in S_2$, then there is a coordinate $k$ such that $\{x_k, y_k, z_k\} = \{0, 1, 2\}$, $\{0,0,1\}$ or $\{0, 0, 2\}$.
\end{enumerate}
\end{definition}
We have a similar construction to \eqref{def:extended product}, but this time we extend a collection of admissible sets rather than cap sets.
\begin{definition} \label{def:meta con}
If $S_0, S_1, S_2$ are admissible sets, and $T \sub \{0, 1,2\}^r$, for each $t = (t_1, \ldots, t_r) \in T$ we define
\[t(S_0, S_1, S_2) = S_{t_1} \times \cdots \times S_{t_r}.\]
Predictably, we then define \[T(S_0, S_1, S_2) = \bigcup_{t \in T} t\bra{S_0, S_1, S_2} = \bigcup_{t \in T} \bra{S_{t_1} \times \cdots \times S_{t_r}}.\]
\end{definition}
\bigbreak
It will not come as a shock that our new definitions of meta-admissible \eqref{def:meta ad} and meta-extendable \eqref{def:meta ext} allow us to use the extended product construction on admissible sets \eqref{def:meta con} to produce new admissible sets.
\begin{lemma}
If $S_0, S_1, S_2 \sub \{0,1,2\}^m$ is a meta-extendable collection of admissible sets, and $T \sub \{0, 1,2\}^r$ is meta-admissible, then $T(S_0, S_1, S_2)$ is an admissible set.
\end{lemma}
\begin{proof}
\textbf{First condition for admissible (pairs):}
Let $x,y \in T(S_0, S_1, S_2)$ be distinct. So we know $x \in S_{t_1} \times \cdots \times S_{t_r}$ and $y \in S_{t'_1} \times \cdots \times S_{t'_r}$ for some $t,t' \in T$. There are 2 cases: $t=t'$ or $t \neq t'$.
\smallbreak
If $t=t'$, we know $x,y \in S_{t_1} \times \cdots \times S_{t_r}$, which is admissible by \eqref{prop: ad prod}.
\smallbreak
Assume $t \neq t'$. By condition (1) of meta-admissible, there is a coordinate $i$ such that $t_i = 0 \neq t_i'$, and $j$ such that $t_j \neq 0 = t'_j$. Without loss of generality, assume $t'_i < t_j$. So
\[x \in S_{t_1} \times \cdots \times S_{0} \times \cdots \times S_{t_j} \times \cdots \times S_{t_r}\]
and 
\[y \in S_{t'_1} \times \cdots \times S_{t'_i} \times \cdots \times S_0 \times \cdots \times S_{t'_r}.\] 
Let $s \in S_0$, $s' \in S_{t'_i}$. Since $t'_i \neq 0$, by property (1) of meta-extendable $s$ has lower weight than $s'$. By the pigeonhole principle there must be a coordinate $k$ such that $s_{k} = 0 \neq s'_k$, as $s$ has more zero entries than $s'$, so $s$ must be zero somewhere $s'$ is not. Similarly, for $s'' \in S_0$, $s''' \in S_{t_j}$, since $t_j \neq 0$ there is a coordinate $\ell$ such that $s''_{\ell} = 0 \neq s'''_{\ell}$. So, it follows that there are coordinates $i', j'$ such that $x_{i'} = 0 \neq y_{i'}$ and $x_{j'} \neq 0 = y_{j'}$, so we have the first condition for admissibility in this case too.
\bigbreak
\textbf{Second condition of admissibility (triples):}
Let $x,y,z \in T(S_0, S_1, S_2)$ be distinct. If $x,y,z$ all come from the same $t \in T$ then we are done, by the direct product construction. So, we need to consider the other cases: $x,y$ come from $t$ and $z$ comes from $t'$ or $x,y,z$ come from distinct $t, t', t''$
\smallbreak
If $x,y$ are from $t$ and $z$ is from $t'$ where $t \neq t'$, then by definition of $T$ being meta-admissible, there is a coordinate $i$ such that $t_i = 0 \neq t'_i$. So, the $i$-th blocks $x_i, y_i$ of $x,y$ are from $S_0$, and the $i$-th block $z_i$ of $z$ is in $S_1 \cup  S_2$. Then by condition (2) of meta-extendable applied to $x_i, y_i, z_i$, there is a coordinate $k$ where $\{x_k, y_k, z_k\} = \{0, 1, 2\}$, $\{0,0,1\}$ or $\{0, 0, 2\}$.
\smallbreak
Finally, if $x,y,z$ from distinct $t, t', t''$ then by condition (2) of meta-admissible, we have a coordinate $k$ where either $\{t_k, t'_k, t''_k\} = \{0, 0, 1\}$, $\{0, 0, 2\}$ or $\{0, 1, 2\}$. 
\smallbreak In the first case, two of $x_k$, $y_k$, $z_k$ are in $S_0$ and the other one is from $S_1 \cup S_2$, and we are done as above, with condition (2) of meta-extendable.
\smallbreak If $\{t_k, t'_k, t''_k\} = \{0, 1, 2\}$, then exactly one of $x_k$, $y_k$, $z_k$ is in each of $S_0$, $S_1$, $S_2$. Then by condition 3 of meta-extendable, we are done.
\end{proof}

\smallbreak
\begin{lemma} \label{metacoll}
There is a meta-extendable collection of admissible sets $S_0, S_1, S_2$, where $S_1$ and $S_2$ are both $I(11,7)$ admissible sets, and $S_0 \sub I(11,3)$ has size 37.
\end{lemma}
\begin{proof}
We use $S_1 = I(11,7)$ from \eqref{computer caps}, and then we take $S_2$ as the set formed by swapping all 1s and 2s in $S_1$. Note that this is still an admissible set, since the doubles condition is unaffected by swapping non-zero coordinates, and if there was a coordinate $k$ where $\{x_k, y_k, z_k\} = \{0,0,1\}$, $\{0,0,2\}$ or $\{0,1,2\}$, then after swapping 1s and 2s we will have $\{x_k, y_k, z_k\} = \{0,0,1\}$, $\{0,0,2\}$ or $\{0,1,2\}$.
\smallbreak
Using a computer search, we can find $S_0 \sub I(11,3)$ such that $|S_0| = 37$ and $S_0, S_1, S_2$ satisfy the conditions for meta-extendable given in \eqref{def:meta ext}. The admissible set $S_0 \sub I(11,3)$ can be found on the author's webpage at \url{http://fredtyrrell.com/cap-sets}.
\end{proof}

\begin{lemma} 
There is an admissible set $T' \sub \{0, 1,2\}^{1562}$ such that $|T'| = 142 \cdot 37 \cdot \binom{11}{7}^{141}$ and all $t \in T'$ have weight 990.
\end{lemma}
\begin{proof}
Let $T = I(142,141)$, which exists by \eqref{lemma:recursive}. Using the meta-extendable collection $(S_0, S_1, S_2)$ from \eqref{metacoll} above, we can then produce a new admissible set $T'=T(S_0, S_1, S_2)$. Since $S_0, S_1, S_2 \sub \{0,1,2\}^{11}$, and $11 \times 142 = 1562$, we see that $T(S_0, S_1, S_2) \sub \{0,1,2\}^{1562}$. 
\smallbreak
Each element $T(S_0, S_1, S_2)$ contains 141 $S_1$ or $S_2$ blocks, and one $S_0$, so has weight $141 \times 7 + 3 = 990$. For each $t \in T$, the set $t(S_0, S_1, S_2)$ has $\binom{11}{7}^{141} \cdot 37$ different vectors, so $|T'| = 142 \cdot 37 \cdot \binom{11}{7}^{141}$.
\end{proof}

Using this new admissible set, we can finally prove the main result of this paper.
\begin{proof}[Proof of theorem \eqref{main}]
We use the admissible set $S=\tilde{I}(6,5)$, which exists by \eqref{lemma:recursive}, and $T' \sub \{0, 1,2\}^{1562}$ from the previous lemma. First, we apply the recursively admissible set $S$ to $(A_0, A_1, A_2)$, the 6 dimensional extendable collection of cap sets from \eqref{extendable caps}, to produce the extendable collection $(B, A_1^6, A_2^6)$, where $B$ has size $6 \cdot 112^5 \cdot 12$. Then we apply $T'$, and our final cap set has size
\[|T'| \cdot |B|^{1562-990} \cdot (112^6)^{990} = \binom{11}{7}^{141} \cdot 6^{572} \cdot 12^{572} \cdot 112^{8800} \cdot 37 \cdot 142.\]
\end{proof}
\begin{remark}
This cap set in $\F_3^{56232}$ has size $|A| \approx 2.21 \times 10^{19455} \approx 10^{10^{4.3}}$, and since $|A|^{1/56232} \approx 2.218021$, this example gives a slight improvement to the lower bound in \eqref{thm:396}.
Similar to the remark \eqref{otherbounds} at the end of section 2, it is possible to find other meta-extendable collections. For example, there is a collection $S_1 = I(11,6)$, $S_2$ is flipped $S_1$ and $S_0 \sub I(11,2)$ of size 20. However, none of these give a better bound than the above.
\end{remark}
\subsection{Summary}
Now we have proved all of the lower bounds in this paper, it seems a good idea to present the results we have achieved, alongside the previous lower bounds. Note that the bounds of Edel in \cite{Edel} and Calderbank and Fishburn in \cite{CalderbankFishburn} both come from what we are now calling the extended product construction.
\smallbreak
\begin{center}
\renewcommand{\arraystretch}{1.5}
\begin{tabular}{ |c|c|c|c|  }
\hline
\textbf{Bound} & \textbf{Construction} & \textbf{Dimension} & \textbf{Appears in} \\ [1ex] 
\hline\hline
2 & $\{0,1\}^n$ & All & Trivial \\[1ex] 
\hline
2.114742\ldots & Maximal cap of size 20 in $\mathbb{F}_3^4$ & 4 & \cite{Pellegrino} \\ [1ex] 
\hline
2.141127\ldots & Maximal cap of size 45 in $\mathbb{F}_3^5$ & 5 & \cite{10.1006/jcta.2002.3261} \\ [1ex] 
\hline
2.195514\ldots & Maximal cap of size 112 in $\mathbb{F}_3^6$ & 6 & \cite{dim6} \\ [1ex] 
\hline
2.210147\ldots & $\tilde{I}(25,24)$ and $I(90,89)$ & 13500 & \cite{CalderbankFishburn} \\ [1ex] 
\hline
2.217389\ldots & $\tilde{I}(8,7)$ and ${I}(10,5)$ & 480 & \cite{Edel} \\ [1ex]
\hline
\hline
2.2175608\ldots & $\tilde{I}(7,6)$ and ${I}(10,6)$ & 420 & \eqref{otherbounds} \\[1ex] 
\hline
2.217950\ldots & $\tilde{I}(7,6)$ and ${I}(11,6)$ & 462 & \eqref{otherbounds} \\[1ex] 
\hline
2.217981\ldots & $\tilde{I}(6,5)$ and ${I}(11,7)$ & 396 & \eqref{thm:396} \\[1ex] 
\hline
2.218021\ldots & Meta-extendable collection & 56232 & \eqref{main} \\[1ex] 
\hline
\end{tabular}
\end{center}

\section{Limits to the admissible set construction}
Given that our new lower bounds have come from finding new admissible sets, it is natural to ask what happens if we continue to find more admissible sets. In particular, we would like to know how much we could improve the lower bound by, if all admissible sets were to exist. Edel gave an answer to this question in the very final section of \cite{Edel}, which we present as the following proposition.
\begin{proposition} \label{prop:asymptotic limit}
 For a collection $(A_0, A_1, A_2)$ of extendable cap sets in $\mathbb{F}_3^{n}$, where $|A_1| = |A_2|$, the best admissible sets are those of the form $I(m,m \alpha)$ where $\alpha = \frac{|A_1|}{|A_0|+|A_1|}$ and $m$ is large. Using the extended product construction, the best constant we could achieve in our asymptotic lower bound is $c = \left(|A_0| + |A_1|\right)^{1/{n}}$.
\end{proposition}
\begin{proof}
Let $\alpha \in (0,1)$. If we apply the admissible set $I(m,m \alpha)$ to $(A_0, A_1, A_2)$, we have a cap set in $\mathbb{F}_3^{nm}$ of size 
\[\binom{m}{m \alpha} \cdot |A_1|^{m \alpha} \cdot |A_0|^{m(1-\alpha)}.\]

For large $m$, we can use the well known asymptotic estimate $\binom{m}{m\alpha} \sim 2^{m h(\alpha)}$, where \[h(x) = -x \log_2(x) - (1-x)\log_2(1-x)\] is the binary entropy function. Taking logs, applying a change of base and removing constants, we see that we want to maximise the function
\[f(x) = x \log\bra{\frac{|A_1|}{|A_0|}} - x\log(x) - (1-x)\log(1-x).\] 
The derivative is given by $f'(x) = \log\bra{\frac{|A_1|}{|A_0|}} + \log(1-x) - \log(x)$, which has its root at $x = \frac{|A_1|}{|A_0|+|A_1|}$. We can then substitute \[w=m \alpha = \frac{m|A_1|}{|A_0|+|A_1|}\] back into our formula from \eqref{lemma:size} for the size of the cap set, giving a cap set in $\mathbb{F}_3^{nm}$ of size $\left(|A_0|+|A_1|\right)^m$. The result follows by taking $nm$-th roots.
\end{proof}
\begin{remark}
For the collection of cap sets in $\mathbb{F}_3^6$ from \eqref{extendable caps}, the best admissible sets are those of the form $I\left(m, \frac{28m}{31}\right)$ for large $m$, and the best asymptotic lower bound we could get using these 6 dimensional cap sets is $\left(124^{1/6}\right)^n = (2.233\ldots)^n$.
\end{remark}
Our experimental evidence, combined with some heuristic arguments (and perhaps a little wishful thinking), lead us to explicitly state the following conjecture, which was implied at the end of \cite{Edel}.
\begin{conjecture} \label{conj1}
For any $m>w>0$, there always exists an $I(m,w)$ admissible set. Therefore, the maximum size of a cap set in $\mathbb{F}_3^n$ is at least $124^{n/6} \approx 2.233^n$.
\end{conjecture}

\begin{remark}
In addition to all admissible sets in dimensions up to 11, we can prove the existence of admissible sets of weight 2 and 3 for all $m$, via a similar construction to \eqref{lemma:recursive}. These, combined with the admissible sets $I(m,0)$, $I(m,1)$, $I(m,m-1)$ and $I(m,m)$, are all of the admissible sets currently known to exist.
\end{remark}
The following table shows the lowest dimension admissible set which would be needed to achieve new bounds, using the extended product construction in \eqref{def:extended product} with the cap sets in $\F_3^6$ from \eqref{extendable caps}.
\smallbreak
\begin{center}
\renewcommand{\arraystretch}{1.4}
\begin{tabular}{ |c|c|c|c|  }
\hline
\textbf{Bound} & \textbf{Admissible Sets} & \textbf{Dimension} \\[1ex] 
\hline\hline
2.220\ldots &  $\tilde{I}( 5 , 4 )$ and $I( 17 , 11 )$  &  510   \\[1ex]
\hline
2.225\ldots &  $\tilde{I}( 3 , 2 )$ and $I( 54 , 41 )$  &  972   \\[1ex]
\hline
2.230\ldots & $I( 311 , 281 )$  &  1866   \\[1ex]
\hline
2.233\ldots & $I( 22948 , 20727 )$  &  137688   \\[1ex]
\hline
2.233076\ldots & $I\left(m, \frac{28m}{31}\right)$ for large $m$ & 6$m$  \\[1ex] 
\hline

\end{tabular}
\end{center}
\bigbreak
While small improvements to our bound may be possible by constructing better admissible sets, we do not expect that a significant increase in the lower bound is possible by simply finding more admissible sets through a computer search. It is perhaps not a huge surprise then that although Edel was able to find an $I(10,5)$ about 20 years ago, the best we have been able to do is an $I(11,7)$. 

\section{The SAT Solver}
\begin{definition}[Boolean satisfiability]
The Boolean satisfiability problem asks whether, given a propositional formula, there exists an assignment of true or false to each variable in the formula such that the formula is true. If this is the case, we say that the formula is \emph{satisfiable}.
\end{definition}

Significant research has gone into finding efficient algorithms for the Boolean satisfiability problem, known as SAT solvers. We used the kissat SAT solver, the winner of the $2020$ SAT competition \cite{SAT}. Most SAT solvers, including kissat, take inputs in a format called conjuctive normal form (CNF).
\begin{definition}[Conjunctive normal form]
A propositional formula is in \emph{conjuctive normal form} if the formula consists of a conjunction of clauses, where each clause is a disjunction of propositional variables or their negations.
\end{definition}

Our use of the SAT solver requires three steps. First, we convert the problem of a particular admissible set existing into a statement in CNF. Once we have defined the problem in CNF, we can use a SAT solver program to check whether this formula is satisfiable. If our formula is satisfiable, the SAT solver returns the assignment of the variables which satisfy the formula, which we convert back into a set of vectors, producing our admissible set. The first step is the interesting one from a mathematical perspective, which we will discuss in this section.

\subsection{CNF algorithm}
\subsubsection{Variables}
We begin by generating the $\binom{m}{w}$ different support sets. In other words, for each $1 \leq i \leq \binom{m}{w}$ we have a different $S_i \sub \{1, \ldots, m\}$, where $|S_i| = w$. We then define a variable for each non-zero coordinate in each support set - let $S_i^k$ correspond to the element $k$ in $S_i$, meaning the $k$ coordinate of the $i$th support set. There are $w \cdot \binom{m}{w}$ such variables $S_i^k$.
\smallbreak
For any pair of support sets, we want to record when a given coordinate is different. So, for each coordinate in the first support set, we define a variable which is true if this coordinate is the same in the second support set as well, and false if it is not. We can represent this variable as $S_{i,j}^k$, corresponding to the pair $S_i, S_j$ of support sets, and the coordinate $k$. This variable $S_{i,j}^k$ is true if coordinate $k$ is in both support sets and $S_i^k = S_j^k$, and is false if coordinate $k$ is only in $S_i$ or $S_i^k \neq S_j^k$.

\subsubsection{Reconciling the pair and triple variables}
We now add constraints, starting with constraints which combine the single variables and the pairwise variables. This is basically just a common sense constraint, to make sure the variables looking at individual coordinates and the variables looking at pairs of coordinates are compatible.

\smallbreak
For any pair $S_i$ and $S_j$ of different support sets, for every coordinate $k$ in both $S_i$ and $S_j$ we ask that either $S_{i,j}^k$ is true or $\{S_i^k, S_j^k\} = \{1,2\}$. An equivalent way to phrase this is as follows: take any pair $x,y$ with different support sets. For each coordinate $k$ where both $x_k$ and $y_k$ are non-zero we add the constraints $(x_k=y_k) \lor (x_k=1) \lor (y_k=1)$ and $(x_k=y_k) \lor (x_k=2) \lor (y_k=2)$. 
\smallbreak
In other words, if the variable which records when $x_k \neq y_k$ and $x_k$, $y_k >0$ is true, then exactly one of the 2 variables which record whether $x_k$ or $y_k$ are 1 is true. This is essentially just saying that $x_k \neq 0 \neq y_k$ and $x_k \neq y_k$ implies $\{x_k, y_k\} = \{1,2\}$.
 
\subsubsection{Checking the condition on triples} We now ensure every triple $x,y,z$ in our set has a coordinate $k$ such that $\{x_k,y_k,z_k\} = \{0,0,1\}$, $\{0,0,2\}$ or $\{0,1,2\}$. This is the triples condition for admissible sets from \eqref{def:admissible}.
\smallbreak
Take any 3 distinct support sets, and check all of the coordinates from $1$ to $m$. If a coordinate is in exactly one of the 3 support sets, we are done and don't need to worry about this triple of support sets - we will be automatically guaranteed a coordinate where our triple is $\{0,0,1\}$ or $\{0,0,2\}$.
\bigbreak
The problem we need to consider is when there is no coordinate in exactly one of the 3 supports. That is, every coordinate in one of the three supports is in at least one of the others. In this case, we need to look at the coordinates in exactly two of the supports, and force one of these coordinates to give us $\{0,1,2\}$. If $k$ is in the support of $x,y$, but not $z$, we can add the condition $x_k \neq y_k$, which would mean $\{x_k, y_k, z_k\} = \{0,1,2\}$.
\smallbreak
We do this for every coordinate in exactly two of the support sets of $x,y,z$, and we then take the disjunction of these conditions, meaning we require this to be true for only one coordinate. This produces a constraint which asks for $\{x_k,y_k,z_k\} = \{0,1,2\}$ in at least one coordinate $k$. So, if this is satisfied for all triples without a coordinate in exactly one of the three supports, we are done, and have an admissible set.

\subsection{Improving the SAT solver} In order for the SAT solver to return an output in a reasonable time, we add more constraints to the problem, reducing the search space of all potential assignments and hence hopefully allowing the SAT solver to work faster. We need to be clever with our choice of extra conditions, to make the algorithm more efficient without turning the problem into one which is impossible. By studying smaller examples of admissible sets, heuristic arguments and a healthy dose of educated guesswork, we tried various combinations of further constraints on our problem. Some made the problem unsatsfiable, and some still did not allow the SAT solver to finish within a reasonable time. However, we used this information to refine our method, and eventually we were successful in producing three new admissible sets. The following additional constraints allowed us to find the admissible sets $I(11,7)$, $I(11,6)$ and $I(10,6)$.

\subsubsection{Extra constraints for $I(11,7)$}
\begin{enumerate}[label=(\roman*)]
    \item Every vector must have the first two non-zero entries different - that is, the first 2 non-zero coordinates are always $(1,2)$ or $(2,1)$.
    \item Every vector must have at least one of each digit 1 and 2 in their 4th, 5th or 6th coordinates - they cannot all be 1 or all be 2.
    \item If the third non-zero entry is in coordinates 1 to 7, it is always a 1.
    \item If the fourth non-zero entry is in coordinates 1 to 7, it is always a 2.
\end{enumerate}
\subsubsection{Extra constraints for $I(11,6)$}
\begin{enumerate}[label=(\roman*)]
    \item Every vector must have the first two non-zero entries different - that is, the first two non-zero coordinates are always $(1,2)$ or $(2,1)$.
    \item Every vector must have at least one of each digit 1 and 2 in the final three non-zero coordinates - they cannot all be 1 or 2.
    \item If the third non-zero entry is in coordinates 1 to 7, it is always a 1.
\end{enumerate}

\subsubsection{Extra constraints for $I(10,6)$}
\begin{enumerate}[label=(\roman*)]
    \item If the second non-zero entry is in the first six coordinates, make it a 1.
    \item If the third non-zero entry is in the first six coordinates, make it a 2.
    \item No vector ends in $(2,2)$.
\end{enumerate}

\begin{remark}
All of the code used in this paper, including the CNF generation code and the code to transform the SAT output back to vectors, can be found on the author's website, at \url{http://fredtyrrell.com/cap-sets} and also at \url{https://github.com/OchenCunningBaldrick/Cap-Sets}. We would like to make it clear that our use of the SAT solver is unlikely to be the most efficient way to produce admissible sets, and we would be very interested in suggestions to improve this aspect of our construction, from those with more experience and knowledge of SAT solvers and other computational methods.
\end{remark}

\section*{Acknowledgments} 
I am extremely grateful to Thomas Bloom for suggesting the problem and supervising my research. I would like to thank Thomas for his support, encouragement and advice throughout my project. In addition, I thank him for many helpful discussions and useful suggestions in preparing this paper.
\smallbreak
Thanks also go to Akshat, Albert, Ittihad, Maria, and Yifan for being an excellent audience for my presentations in the Maths Institute, where they provided me with a valuable opportunity to share and discuss my research.
\smallbreak
Finally, I would like to thank Alex, Anubhab, Chris, Claire, Flora, Jess and Zach for allowing me to try and explain my work to them, with varying levels of success.
\smallbreak
The work in this paper was completed while the author was employed as a Summer Project Intern in the Mathematical Institute, University of Oxford, under the supervision of Dr Thomas Bloom.

\bibliographystyle{amsplain}
\bibliography{caps}

\begin{dajauthors}
\begin{authorinfo}[Fred]
  Fred Tyrrell\\
  Fry Building, School of Mathematics\\
  Bristol, UK\\
  fred\imagedot{}tyrrell\imageat{}bristol\imagedot{}ac\imagedot{}uk\\
  \url{http://fredtyrrell.com}
\end{authorinfo}
\end{dajauthors}

\end{document}